\documentclass[a4paper,11pt]{amsart}
\usepackage[matrix,arrow,tips,curve]{xy}
\usepackage[english]{babel}
\usepackage{amsmath}
\usepackage{amssymb}
\usepackage{tikz}
\usepackage{mathrsfs}
\usepackage{enumerate}
\usepackage{graphicx}
\usepackage{hyperref}
\usetikzlibrary{trees}
\usetikzlibrary{arrows}

\newtheorem*{thm*}{Theorem}

\theoremstyle{definition}

\newtheorem{teo}{Theorem}[section]
\newtheorem{lema}[teo]{Lemma}

\newtheorem{cor}[teo]{Corollary}

\newtheorem{ex}{Example}[section]

\newtheorem{rem}{Remark}[section]

\definecolor{wwwwww}{rgb}{0.4,0.4,0.4}

\hypersetup{pdfpagemode=UseNone}
\hypersetup{pdfstartview=FitH}
		
\begin{document}

\title[\resizebox{6.1in}{!}{Volume estimates and classification theorem for constant weighted mean curvature hypersurfaces}]{Volume estimates and classification theorem for constant weighted mean curvature hypersurfaces}

\author[Saul Ancari]{Saul Ancari}\thanks{This study was financed in part by the Coordenação de Aperfeiçoamento de Pessoal de Nível Superior - Brasil (CAPES) Finance Code 001 }
\address{\sc Saul Ancari\\
Instituto de Matem\'atica e Estat\'istica, Universidade Federal Fluminense\\
24210-201 Niter\'oi, Rio de Janeiro\\ Brazil}
\email{sa\_ancari@id.uff.br}

\author[Igor Miranda]{Igor Miranda}
\address{\sc Igor Miranda\\
Instituto de Matem\'atica e Estat\'istica, Universidade Federal Fluminense\\
24210-201 Niter\'oi, Rio de Janeiro\\ Brazil}
\email{igor\_miranda@id.uff.br}


\begin{abstract}
In this paper, we prove a classification for complete embedded constant weighted mean curvature hypersurfaces $\Sigma\subset\mathbb{R}^{n+1}$. We characterize the hyperplanes and generalized round cylinders by using an intrinsic property on the norm of the second fundamental form. Furthermore, we prove an equivalence of  properness, finite weighted volume and exponential volume growth for submanifolds with weighted mean curvature of at most linear growth.
\end{abstract}

\maketitle 
\section{Introduction}
In the mean curvature flow theory, one of the main problems is to understand possible singularities that the flow goes through. Self-shrinkers play  an important role in this theory since they are singularity models for the flows . These hypersurfaces satisfy the following mean curvature condition
\[
H=\frac{\langle x , \nu \rangle}{2}
\]
where $H$, $x$ and $\nu$ stand for the mean curvature of $\Sigma$, the position vector in $\mathbb{R}^{n}$ and the unit normal vector of $\Sigma$, respectively. Another characterization of the self-shrinkers is that they are critical points of the weighted area functional
\stepcounter{thm}
\begin{eqnarray}\label{11}
F(\Sigma)=\int_\Sigma e^{-\frac{|x|^2}{4}}dv.
\end{eqnarray}

There is a great interest in studying two-sided smooth hypersurfaces $\Sigma\subset\mathbb{R}^{n+1}$ which are critical points of the functional (\ref{11}) for variations $G:(-\varepsilon,\varepsilon)\times\Sigma \rightarrow \mathbb{R}^{n+1}$ that preserve enclosed weighted volume. These variations can be represented by functions $u:\Sigma \rightarrow \mathbb{R}$ defined by 
\[
u(x)=\langle \partial_t G(0,x),\nu(x)\rangle
\]
such that $\int_\Sigma u\  e^{-|x|^2/4}dv =0$, where $\nu$ is the normal vector of $\Sigma$. It is well known that these hypersurfaces satisfy the following condition
\begin{eqnarray*}
H=\frac{\langle x, \nu\rangle}{2}+\lambda,
\end{eqnarray*}
where $\lambda \in \mathbb{R}$. Such hypersurfaces are known as constant weighted mean curvature hypersurfaces. Throughout this paper, whenever $\Sigma$ satisfies the mean curvature condition above, they will be called CWMC hypersurfaces and $\lambda$ denotes the weighted mean curvature. The study of such hypersurfaces arises in geometry and probability as solutions to the Gaussian isoperimetric problem.\\
Here are some examples of CWMC hypersurfaces.
\ex Any self-shrinker is a CWMC hypersurface with $\lambda=0$.
\ex All hyperplanes in $\mathbb{R}^{n+1}$ are CWMC hypersurfaces with $\lambda=\pm\frac{d}{2}$, where $d$ denotes the distance from the hyperplane to the origin and the sign depends on the orientation. Indeed, let $\Sigma\subset \mathbb{R}^{n+1}$ be a hyperplane and $x_0\in \Sigma$ such that $d(\Sigma,0)=d(x_0,0)=d$. This implies that $\pm d=\langle x_0, \nu\rangle$, where $\nu$ is the normal vector of $\Sigma$. Since $\Sigma= \lbrace x\in \mathbb{R}^{n+1}; \langle x, \nu \rangle= \pm d \rbrace$ and the mean curvature of $\Sigma$ vanishes, we have
\[
H=0=\frac{\langle x,\nu \rangle}{2} \pm \frac{d}{2}.
\]
Therefore, $\Sigma$ is a CWMC hypersurface with $\lambda=\pm d/2$, as asserted.

\ex The spheres centered at the origin with radius $\sqrt{\lambda^2+2n}-\lambda$ are CWMC hypersurfaces in $\mathbb{R}^{n+1}$.
\ex The cylinders $S^k_r(0)\times \mathbb{R}^{n+k}$ with radius $\sqrt{\lambda^2+2k}-\lambda$ are also CWMC hypersurfaces in $\mathbb{R}^{n+1}$.\\

In \cite{huisken1990asymptotic}, Huisken proved that the spheres $S^n(\sqrt{2n})$ are the only compact self-shrinkers with non-negative mean curvature of dimension $n\geq 2$. Huisken \cite{huisken54local} also showed that the generalized cylinders are the only complete self-shrinkers in $\mathbb{R}^{n+1}$ with non-negative mean curvature, polynomial volume growth and such that the norm of the second fundamental form is bounded. Colding and Minicozzi \cite{colding2012generic} generalized Huisken's classification by removing the boundness condition on the second fundamental form. Rimoldi \cite{rimoldi2014classification} proved the same classification theorem as Colding-Minicozzi, replacing the polynomial volume growth by a assumption on the integrability of $|A|$. More specifically
\begin{teo}\label{ti3} $S^k(\sqrt{2n})\times \mathbb{R}^{n-k}$ , $0 \leq k \leq n$, are the only smooth complete embedded self-shrinkers in $\mathbb{R}^{n+1}$ with $H\geq 0$ and $|A|\in L^2_f(\Sigma)$.
\end{teo}
Recently, there has been much interest around classification results for CWMC hypersurfaces. For instance, Guang \cite{guang2018gap} proved a gap theorem for CWMC hypersurfaces, showing that if the hypersurface has a bound condition on $|A|$, then it must be a generalized cylinder. In \cite{heilman2017symmetric}, Heilman proved a generalization of Colding-Minicozzi theorem for $\lambda>0$.
For what follows, it is important to recall a result obtained by Cheng and Wei \cite{cheng2018complete} for CWMC hypersurfaces. More precisely, they establish the following result.

\begin{teo}\label{ti0} $S^k\times \mathbb{R}^{n-k}$ , $0 \leq k \leq n$, are the only complete embedded CWMC hypersurfaces with polynomial volume growth in $\mathbb{R}^{n+1}$ satisfying $H-\lambda\geq 0$ and $\lambda\left(\text{tr}A^3(H-\lambda)+\frac{|A|^2}{2}\right)\leq 0$.
\end{teo}
The hypothesis of polynomial volume growth is used to show that weighted integrals converge in order to justify integration by parts. In fact this hypothesis also implies the hypersurface is proper. In this paper, we will replace the polynomial volume growth assumption considered by Cheng-Wei by an intrinsic condition. More specifically, we prove the following result.
\begin{teo}\label{ti1}
Let $\Sigma\subset \mathbb{R}^{n+1}$ be a complete embedded CWMC hypersurface. Suppose that $\Sigma$ satisfies the following properties:\\
(i) $H-\lambda\geq 0$;\\
(ii) $\lambda\left( \text{tr}A^3(H-\lambda)+\frac{|A|^2}{2}\right)\leq 0$;\\
(iii) $\frac{1}{k^2}\int_{B^\Sigma_{2k}(p)\setminus B^\Sigma_k(p)} |A|^2e^{-f}\rightarrow 0$, when $k\rightarrow \infty$, for a fixed point $p\in \Sigma$.\\
Then $\Sigma$ must be either a hyperplane or $S^{k}_r(0)\times \mathbb{R}^{n-k}$, $1\leq k \leq n$.
\end{teo} 

\rem Note that Theorem \ref{ti1} generalizes Theorem \ref{ti3} proved by Rimoldi.\\

The second result of this paper provides a relation between the exponential volume growth condition and properness for submanifolds with weighted mean curvature of at most linear growth. For self-shrinkers, Ding and Xin \cite{ding2013volume} proved that the properness condition implies Euclidean volume growth. Cheng and Zhou \cite{cheng2013volume} proved that properness, Euclidean volume growth, polynomial volume growth and finite weighted volume are all equivalent for self-shrinkers. Alencar and Rocha \cite{alencar2018stability} showed that $|\overrightarrow{H}_f|<\infty$ and finite weighted volume imply properness (see section \ref{s2} for the definition of $\overrightarrow{H}_f$). Also, a consequence from one of their results is that for $f=|x|^2/4$, $\sup \langle \overrightarrow{H_f},\overline{\nabla} f\rangle<\infty$ and properness imply finite weighted volume and polynomial volume growth. Recently, Cheng, Vieira and Zhou \cite{cheng2019volume} proved that for hypersurfaces in $\mathbb{R}^{n+1}$ with the norm of the weighted mean curvature bounded, properness is equivalent to polynomial volume growth. Here is our second result.

\begin{teo}\label{ti2} For any complete n-dimensional immersed submanifold $\Sigma$ in $\mathbb{R}^{n+p}$, $p\geq 1$, satisfying
\[
|\overrightarrow{H}_f|\leq a_1r+a_0\ \text{on}\ \Sigma\cap B_r(0),\ \forall r>0
\]
for some $a_0,a_1\geq 0$, where $f=|x|^2/4$, the following statements are equivalent:\\
(i) $\Sigma$ properly immersed on $\mathbb{R}^{n+p}$;\\
(ii) There exist constants $C, \overline{a}_0,\overline{a}_1,\overline{a}_2$, with $\overline{a}_2<\frac{1}{4}$, such that
\begin{eqnarray*}
V(B_r(0)\cap \Sigma)\leq Ce^{\overline{a}_2r^2+\overline{a}_1r+\overline{a}_0};
\end{eqnarray*}
(iii)$\int_\Sigma e^{-f}<\infty$.
\end{teo}
\rem Recall that a submanifold $\Sigma^n\subset \mathbb{R}^{n+p}$ that satisfies $\overrightarrow{H}=\frac{x^\perp}{2}$ is called self-expander. A translating soliton is a submanifold in $\mathbb{R}^{n+p}$ that satisfies $\overrightarrow{H}=z^\perp$, for $z\in\mathbb{R}$ fixed. Note that for self-expanders $|\overrightarrow{H}_f|=|x^\perp|$  and for translating solitons $|\overrightarrow{H}_f|=|z^\perp +\frac{x^\perp}{2}|$. Therefore, if $\Sigma^n\subset \mathbb{R}^{n+p}$ is a complete properly immersed self-expander or translating soliton, then there exist constants $C,\overline{a}_0,\overline{a}_1,\overline{a}_2$, with $\overline{a}_2<\frac{1}{4}$, such that
\[
V(B_r(0)\cap \Sigma)\leq Ce^{\overline{a}_2r^2+\overline{a}_1r+\overline{a}_0}.
\]

\indent Finally, a corollary of Theorem \ref{ti1} and Theorem \ref{ti2} is the following.

\begin{cor}\label{ci1} Let $\Sigma\subset \mathbb{R}^{n+1}$ be a complete properly embedded CWMC hypersurface. If $H-\lambda\geq 0$ and $\lambda\left( \text{tr}A^3(H-\lambda)+\frac{|A|^2}{2}\right)\leq 0$, then $\Sigma$ must be a hyperplane or $S^{k}_r(0)\times \mathbb{R}^{n-k}$, $1\leq k \leq n$.
\end{cor}

This work is divided into four sections. In section \ref{s2}, we recall some notations, basic tools and key formulas for CWMC hypersurfaces. In section \ref{s3}, we will prove the classification theorem. In section \ref{s4}, we will prove the equivalence between exponential volume growth and properness. In section \ref{s5}, we provide an application of the main results of this paper.\\

\textbf{Acknowledgement}: The authors would like to thank professor Detang Zhou for his support, suggestions and encouragement throughout this work. We also want to thank professor Ernani Ribeiro Jr. for his helpful advices. 
\section{Preliminaries}\label{s2}
In this section, we will establish some notations and recall some definitions and basic results.\\

 Let us denote by $(\overline{M}^n,\overline{g},e^{-f}dv)$ a smooth measure metric space, which is a n-dimensional Riemannian manifold $(\overline{M},\overline{g})$ endowed with $e^{-f}dv$ volume form, where $f$ is a smooth function over $\overline{M}$ and $dv$ is the volume form induced by the metric $\overline{g}$. Throughout this work, whenever we integrate, we will omit $dv$. We will also denote the connection of $(\overline{M},\overline{g})$ by $\overline{\nabla}$.\\

 For what follows, recall that the drifted Laplacian over $(\overline{M}, \overline{g})$ is
\[
\overline{\Delta}_fu=\overline{\Delta} u - \langle \overline{\nabla}f, \overline{\nabla} u\rangle.
\]
If $\Sigma \subset \mathbb{R}^{n+1}$ is a hypersurface and $f=|x|^2/4$, we denote the drifted Laplacian by
\[
\mathcal{L}=\Delta -\frac{1}{2}\langle x,\nabla\rangle.
\]
This operator is self-adjoint over $L^2(\Sigma, g, e^{-f}d\sigma)$. More precisely,
\begin{lema}\label{pl1}If $\Sigma\subset \mathbb{R}^{n+1}$, $u\in C^1_0(\Sigma)$ and $v\in C^2(\Sigma)$, then 
\[
\int_\Sigma u\mathcal{L}ve^{-f}=-\int_\Sigma \langle \nabla u, \nabla v\rangle e^{-f}.
\]
\end{lema}
\begin{proof}
For a detailed proof, see \cite{colding2012generic}.
\end{proof}
 Let $\Sigma$ be a submanifold immersed on $\overline{M}$, endowed with the metric g induced by $\overline{g}$. We will denote by $\nabla$, $\Delta$ and $d\sigma$, the connection, Laplacian  and volume form, respectively. The second fundamental form of $(\Sigma,g)$ at $p\in \Sigma$ is defined as 
\[
A(X,Y)= (\overline{\nabla}_XY)^\perp,
\]
where $X,Y\in T_p\Sigma$. The mean curvature vector $\overrightarrow{H}$ of $\Sigma$ at $p$ is defined as 
\begin{eqnarray*}
\overrightarrow{H}=\text{tr}A.
\end{eqnarray*}
The weighted mean curvature vector of $\Sigma$ at $p$ is defined as
\[
\overrightarrow{H_f}=\overrightarrow{H}+(\overline{\nabla}f)^\perp.
\]
If $\Sigma$ is a hypersurface immersed in $\overline{M}^{n+1}$, we will denote by $h_{ij}=\langle A(e_i,e_j),\nu\rangle$, where $\lbrace e_i\rbrace$ is an orthonormal basis of $T_p\Sigma$. The mean curvature of $\Sigma$ is defined as $\overrightarrow{H}=-H\nu$ and the weighted mean curvature as $\overrightarrow{H_f}=-H_f\nu$. Throughout this work, whenever $\Sigma\subset \mathbb{R}^{n+p}$, we will be considering the smooth measure metric space $(\mathbb{R}^{n+p}, g, e^{-f}d\sigma)$, with $f=\frac{|x|^2}{4}$. 
A hypersurface $\Sigma\subset\mathbb{R}^{n+1}$ is CWMC hypersurface if 
\[
H_f=\lambda.
\]
In particular, if $\lambda=0$, $\Sigma$ is a self-shrinker.\\

For CWMC hypersurfaces, the following equations will be needed to prove the classification theorem.
\begin{lema}\label{l1}If $\Sigma\subset\mathbb{R}^{n+1}$ is a CWMC hypersurface, then
\[
\mathcal{L}(H-\lambda)+\left(|A|^2-\frac{1}{2}\right)(H-\lambda)= \frac{\lambda}{2}
\]
and
\[
\mathcal{L}|A|+\left(|A|^2-\frac{1}{2}\right)|A|= \frac{|\nabla A|^2-|\nabla |A||^2}{|A|}-\frac{\lambda \text{tr}A^3}{|A|}.
\]
\end{lema}
\begin{proof}
For a detailed proof, see \cite{guang2018gap}.
\end{proof}
\section{A Classification for CWMC hypersurfaces}\label{s3}
In this section, we prove Theorem \ref{ti1}. The main idea of proving this theorem is to use a technique similar to that used by Tasayco and Zhou in \cite{tasayco2017uniqueness} to prove that $|A|=C(H-\lambda)$, where C is a constant.\\

\noindent\textit{Proof of Theorem \ref{ti1}}.
For $\lambda\leq 0$, from Lemma \ref{l1} we have
\[
\mathcal{L}(H-\lambda)+(|A|^2-\frac{1}{2})(H-\lambda)=\frac{\lambda}{2}\leq 0.
\]
Since $H-\lambda\geq 0$, by the maximum principle we can conclude that either $H-\lambda=0$ or $H-\lambda>0$. If $H-\lambda=0$, then from Lemma \ref{l1} we conclude that $\lambda=0$, which implies that $\Sigma$ is a self-shrinker. Moreover, Colding-Minicozzi proved in \cite{colding2012generic} that a self-shrinker such that $H=0$ has to be a hyperplane. If $\lambda>0$ and $H-\lambda=0$ at some point $p\in \Sigma$, from hypothesis (ii)
\[
0\geq \lambda\left( \text{tr}A^3(H-\lambda)+\frac{|A|^2}{2}\right)=\frac{\lambda|A|^2}{2}
\]
at $p \in \Sigma$. This implies that $|A|(p)=0$, but this contradicts the fact that $H(p)>0$. Finally, we only we need to consider the case $H-\lambda>0$.\\

By adapting Tasayco-Zhou's lemma \cite{tasayco2017uniqueness}, we will demonstrate that either $|A|=0$ or $|A|=C(H-\lambda)$, for $C>0$. Consider the functions $u=H-\lambda$ and $v= \sqrt{|A|^2+\varepsilon}$. Computing $\mathcal{L}(u)$ and $\mathcal{L}(v)$, we get
\stepcounter{thm}
\begin{eqnarray}\label{7}
\mathcal{L}u+\left(|A|^2-\frac{1}{2}\right)u= \frac{\lambda}{2}
\end{eqnarray}
and
\[
\mathcal{L}v+\left(|A|^2-\frac{1}{2}\right)v= \frac{|\nabla A|^2-|\nabla v|^2}{v}+\left(|A|^2-\frac{1}{2}\right)\frac{\varepsilon}{v}-\frac{\lambda \text{tr}A^3}{v}.
\]
Since
\stepcounter{thm}
\begin{eqnarray}\label{8}
\frac{|\nabla A|^2-|\nabla v|^2}{v}\geq 0,\ \ \mathcal{L}v+\left(|A|^2-\frac{1}{2}\right)v\geq -\frac{\varepsilon}{2v}-\frac{\lambda \text{tr}A^3}{v}.
\end{eqnarray}
Let us consider $w=\frac{v}{u}$. Thus, from (\ref{7}) we get
\begin{eqnarray*}
\mathcal{L}v&=&w\mathcal{L}u +2\langle \nabla w, \nabla u\rangle +u\mathcal{L}w\\
			&=&w\left(\frac{\lambda}{2}-|A|^2+\frac{1}{2}\right)+2\langle \nabla w, \nabla u \rangle+u\mathcal{L}w.
\end{eqnarray*}
By (\ref{8}), we have
\[
u\mathcal{L}w\geq -\frac{1}{v}\left(\frac{\varepsilon}{2}+\lambda \text{tr}A^3\right)-\frac{\lambda w}{2}-2\langle \nabla w, \nabla u\rangle.
\]
Using hypothesis (ii), it is possible to conclude that
\stepcounter{thm}
\begin{eqnarray}\label{9}
\frac{u}{v}\left(-\frac{\varepsilon}{2}-\lambda \text{tr}A^3 \right)-\frac{v\lambda}{2}\geq -\frac{\varepsilon H}{2v}.
\end{eqnarray}
From the inequality above, we obtain
\stepcounter{thm}
\begin{eqnarray}\label{10}
\mathcal{L}w\geq -\frac{\varepsilon H}{2vu^2}-2\langle \nabla w, \nabla \log u\rangle.
\end{eqnarray}
For a function $\varphi\in C_0^\infty(\Sigma)$, using integration by parts and (\ref{10}), we get
\begin{eqnarray*}
\int_\Sigma \varphi^2|\nabla w|^2e^{-f}&=&-\int_\Sigma \varphi^2 w \mathcal{L}w e^{-f}-\int_\Sigma 2\varphi w\langle \nabla \varphi, \nabla w\rangle e^{-f}\\
									   &\leq & 2\int_\Sigma \varphi^2w\langle \nabla w, \nabla \log u\rangle e^{-f}+\frac{\varepsilon}{2}\int_\Sigma \frac{\varphi^2wH}{vu^2}e^{-f}-\int_\Sigma2\varphi w\langle \nabla \varphi,\nabla w\rangle e^{-f}\\
									   &=& 2\int_\Sigma \langle \varphi \nabla w, \varphi w \nabla \log u-w\nabla \varphi\rangle e^{-f}+\frac{\varepsilon}{2}\int_\Sigma \frac{\varphi^2H}{u^3}e^{-f}\\
									   &\leq & \frac{1}{2}\int_\Sigma \varphi^2 |\nabla w|^2 e^{-f}+2\int_{\Sigma} w^2|\varphi\nabla \log u-\nabla \varphi|^2 e^{-f}+\frac{\varepsilon}{2}\int_\Sigma \frac{\varphi^2H}{u^3}e^{-f}.
\end{eqnarray*}
Therefore,
\begin{eqnarray*}
\int_\Sigma \varphi^2|\nabla w|^2\leq 4 \int_\Sigma w^2|\varphi\nabla \log u-\nabla \varphi|^2e^{-f}+\varepsilon\int_\Sigma \frac{\varphi^2H}{u^3}e^{-f}.
\end{eqnarray*}
Choosing $\varphi=\psi u$, $\psi\in C_0^\infty(\Sigma)$, we have
\begin{eqnarray*}
\int_\Sigma \psi^2 u^2|\nabla w|^2 e^{-f}&\leq & 4 \int_\Sigma v^2 |\nabla \psi|^2e^{-f}+\varepsilon\int_\Sigma \psi^2 e^{-f}+\varepsilon\lambda\int_\Sigma \frac{\psi^2}{u}e^{-f}.
\end{eqnarray*}
For $\lambda\geq 0$, choosing $\varepsilon = 0$ we obtain
\begin{eqnarray*}
\int_\Sigma \psi^2 u^2|\nabla w|^2 e^{-f}&\leq & 4 \int_\Sigma v^2 |\nabla \psi|^2e^{-f}.
\end{eqnarray*}
Consider a sequence $\psi_k\in C_0^\infty(\Sigma)$, such that $\psi_k=1$ in $B_k^\Sigma(p)$, $\psi_k=0$ in $\Sigma\setminus B_{2k}^\Sigma(p)$ and $|\nabla \psi_k|\leq 1/k$ for every $k$, we have
\begin{eqnarray*}
\int_\Sigma \psi_k^2 u^2|\nabla w|^2 e^{-f}&\leq & 4 \int_{B_{2k}^\Sigma(p)\setminus B_k^\Sigma(p)} v^2 |\nabla \psi_k|^2e^{-f}\\
										   &\leq & \frac{4}{k^2} \int_{B_{2k}^\Sigma(p)\setminus B_k^\Sigma(p)} v^2 e^{-f}\\
										   &= & \frac{4}{k^2} \int_{B_{2k}^\Sigma(p)\setminus B_k^\Sigma(p)} |A|^2 e^{-f}.
\end{eqnarray*}
By the monotone convergence theorem and hypothesis (iii), we get
\[
\int_\Sigma u^2\left|\nabla \left(\frac{|A|}{H-\lambda}\right)\right|^2e^{-f}=0,
\]
which implies that $|A|=C(H-\lambda)$, for a constant $C>0$. 

For $\lambda<0$, we have 
\[
\int_\Sigma \psi^2 u^2|\nabla w|^2 e^{-f}\leq 4 \int_\Sigma v^2 |\nabla \psi|^2e^{-f}+\varepsilon\int_\Sigma \psi^2 e^{-f}.
\]
As in the other case, consider a sequence $\psi_k\in C_0^\infty(\Sigma)$, such that $\psi_k=1$ in $B_k^\Sigma(p)$, $\psi_k=0$ in $\Sigma\setminus B_{2k}^\Sigma(p)$ and $|\nabla \psi_k|\leq 1/k$ for every $k$, hence we get
\begin{eqnarray*}
\int_\Sigma \psi_k^2 u^2|\nabla w|^2 e^{-f}&\leq & 4 \int_\Sigma v^2 |\nabla \psi_k|^2e^{-f}+\varepsilon\int_\Sigma \psi_k^2 e^{-f}\\
										   &\leq & \frac{4}{k^2}\int_{B_{2k}^\Sigma(p)\setminus B_k^\Sigma(p)} |A|^2 e^{-f} +\frac{4\varepsilon}{k^2}\int_{B_{2k}^\Sigma(p)\setminus B_k^\Sigma(p)} e^{-f}+\varepsilon\int_{B_{2k}^\Sigma(p)} e^{-f}.
\end{eqnarray*}
Choosing $\varepsilon = \left(k\int_{B_{2k}^\Sigma(p)} e^{-f}\right)^{-1}$, we have
\begin{eqnarray*}
\int_\Sigma \psi_k^2 u^2|\nabla w|^2 e^{-f}\leq \frac{4}{k^2}\int_{B_{2k}^\Sigma(p)\setminus B_k^\Sigma(p)} |A|^2 e^{-f} +\frac{4}{k^3}+\frac{1}{k}.
\end{eqnarray*}
Hence, by hypothesis (iii)
\begin{eqnarray*}
\lim_{k\rightarrow \infty} \int_\Sigma \psi_k^2 u^2|\nabla w|^2 e^{-f}=0.
\end{eqnarray*}
If the set
\[
\mathcal{A}=\lbrace p\in \Sigma; |A|(p)=0\rbrace
\]
is not empty, consider $\mathcal{B}=\Sigma\setminus \mathcal{A}$. Since $\mathcal{B}$ is an open set, let $p\in \mathcal{B}$ and $B^\Sigma_p(r)\subset \mathcal{B}$. For k sufficiently large, $B^\Sigma_p(r)\subset \text{supp}\psi_k$ and $\psi_k=1$ in $B^\Sigma_p(r)$. Hence
\[
\lim_{k\rightarrow \infty} \int_{B^\Sigma_p(r)}  u^2|\nabla w|^2 e^{-f}=0.
\]
By the dominated convergence theorem, we conclude that $|A|/(H-\lambda)$ is constant in $B^\Sigma_p(r)$. Since p is arbitrary, it is possible to conclude that $|A|/(H-\lambda)$ is constant in $\mathcal{B}$. Since $\mathcal{A}\neq \emptyset$, using a continuity argument, we conclude that $|A|=0$. If $\mathcal{A}= \emptyset$, by the dominated convergence theorem
\[
\int_\Sigma u^2\left|\nabla \left(\frac{|A|}{H-\lambda}\right)\right|^2e^{-f}=0,
\]
which implies $|A|=C(H-\lambda)$ for a constant $C>0$. Hence, when $H-\lambda>0$, we conclude that either $|A|=0$ or $|A|=C(H-\lambda)$ for a constant $C>0$. If $|A|=0$, then $\Sigma$ is a hyperplane. Otherwise, since $|A|=C(H-\lambda)$
\begin{eqnarray*}
\mathcal{L}|A|&=&\frac{|A|}{H-\lambda}\mathcal{L}(H-\lambda)\nonumber\\
			  &=& \frac{|A|\lambda}{2(H-\lambda)} +\left(\frac{1}{2}-|A|^2\right)|A|.
\end{eqnarray*}
On the other hand
\begin{eqnarray*}
\mathcal{L}|A| = \left(\frac{1}{2}-|A|^2\right)|A| +\frac{|\nabla A|^2 - |\nabla |A||^2}{|A|}-\frac{\lambda \text{tr}A^3}{|A|}.
\end{eqnarray*}
Hence, from the equations above we have
\begin{eqnarray*}
\frac{|\nabla A|^2 - |\nabla |A||^2}{|A|}&=& \frac{|A|\lambda}{2(H-\lambda)}+\frac{\lambda \text{tr}A^3}{|A|}\\
										 &=& \frac{\lambda}{(H-\lambda)|A|}\left(\text{tr}A^3(H-\lambda)+\frac{|A|^2}{2}\right).
\end{eqnarray*}
Using the hypothesis $\lambda(\text{tr}A^3(H-\lambda)+\frac{|A|^2}{2})\leq 0$ and the equality above, we conclude that 
\stepcounter{thm}
\begin{eqnarray}\label{12}
|\nabla |A||=|\nabla A|.
\end{eqnarray}
Fixing $p\in\Sigma$ and $\lbrace E_i\rbrace_{1\leq i \leq n}$ a orthonormal basis for $T_p\Sigma$, (\ref{12}) implies that for each $k$ there exists a constant $C_k$ such that
\[
h_{ijk}= C_kh_{ij}
\]
for all $i,j$. Considering a base such that $h_{ij}=\lambda_i\delta_{ij}$, by the Codazzi equation, we have
\[
h_{ijk}=0
\]
unless $i=j=k$. If $\lambda_i\neq 0$ and $i\neq j$ then
\[
0=h_{iij}=C_j\lambda_i.
\]
It follows that $C_j=0$. Hence, if the rank of the matrix $(h_{ij})$ is at least two at p, then $\nabla A(p)=0$. To show that $\nabla A=0$, let us fix $q\in\Sigma$ and suppose that $\lambda_1(q)$ and $\lambda_2(q)$ are the largest eigenvalues of $(h_{ij})(q)$. Define the following set
\[
\Lambda=\{q\in\Sigma;\lambda_1(q)=\lambda_1(p),\lambda_2(q)=\lambda_2(p)\}.
\]
Using the continuity of the $\lambda_i's$, it is possible to prove that the set $\Lambda$ is open and closed. Since $p\in \Lambda$ and $\Sigma$ is connected, $\Lambda=\Sigma$. Therefore, $\nabla A=0$ everywhere on $\Sigma$. Hence, $\Sigma$ is a isoparametric hypersurface and by a theorem proved by Lawson in \cite{lawson1969local}, $\Sigma$ must be $S^k(r)\times \mathbb{R}^{n-k}$ with $2\leq k \leq n$.\\
If the rank of the matrix $(h_{ij})$ is one, then
\[
H^2=|A|^2=C^2(H-\lambda)^2.
\]
From this equation, $H$ must be constant. Moreover, from
\[
|\nabla A|=|\nabla|A||=C|\nabla H| = 0,
\]
we conclude that $\Sigma$ is isoparametric and by Lawson's result $\Sigma$ must be $S^1(r)\times\mathbb{R}^{n-1}$.
\begin{flushright}
$\Box$
\end{flushright}

\section{Volume estimates for submanifolds with weighted mean curvature of at most linear growth }\label{s4}

In this section, we show some volume estimates and prove the second main result of this paper.\\

Let us consider $D_r\subset M$ as the level set
\begin{eqnarray*}
D_r=\{x\in M; 2\sqrt{f}<r\},
\end{eqnarray*}
and $V(r)$ the volume of $\overline{D}_r$.
\begin{teo}\label{0.1} Let $(M^n,g,e^{-f}dv)$ be a complete non-compact smooth measure metric space, with $f:M\rightarrow \mathbb{R}$ a proper function on M. If $|\nabla f|^2\leq f$ and $\Delta_ff+f\leq a_2r^2+a_1r+a_0$ on $D_r$ for all $r>0$, where $a_0,a_1,a_2$ are constants, then $\text{Vol}_f(M)<\infty$ and for $\varepsilon>0$ arbitrary
\stepcounter{thm}
\begin{eqnarray}
V(r)\leq Ce^{\varepsilon (a_2r^2+a_1r+a_0)+\frac{r^2}{4e^{\varepsilon}}}.
\end{eqnarray}
\end{teo}
\begin{proof} Let us define
\begin{eqnarray*}
I(t)=\frac{1}{t^{\kappa(r)}}\int_{\bar{D}_r}e^{-\frac{f}{t}}dv
\end{eqnarray*}
for all $t>0$, where $\kappa(r)=a_2r^2+a_1r+a_0$. Since f is proper, $I$ is well defined.\\
Computing the derivative of $I$, we get
\stepcounter{thm}
\begin{eqnarray}
I^\prime(t)=t^{-\kappa(r)-1}\int_{\bar{D}_r}e^{-\frac{f}{t}}\left(\frac{f}{t}-\kappa(r)\right)dv.
\end{eqnarray}
On the other hand
\begin{eqnarray*}
\int_{\bar{D}_r} \text{div}\left(e^{-\frac{f}{t}}\nabla f\right)dv&=& \int_{\bar{D}_r}e^{-\frac{f}{t}}\left(\Delta f -\frac{|\nabla f|^2}{t}\right)dv\\
														   &\leq & \int_{\bar{D}_r}e^{-\frac{f}{t}}\left(|\nabla f|^2-f+\kappa(r)-\frac{|\nabla f|^2}{t}\right)dv\\
														   &\leq & \int_{\bar{D}_r}e^{-\frac{f}{t}}\left(\frac{(t-1)}{t}f-f+\kappa(r)\right)dv,\ \ t\geq 1\\
														   &=& \int_{\bar{D}_r}e^{-\frac{f}{t}}\left(\kappa(r)-\frac{f}{t}\right)dv\\
														   &=& -I^\prime(t) t^{\kappa(r)+1}.
\end{eqnarray*}
Therefore, we have
\begin{eqnarray*}
I^\prime(t)\leq -t^{-\kappa(r)-1}\int_{\bar{D}_r}\text{div}\left(e^{-\frac{f}{t}}\nabla f\right)dv.
\end{eqnarray*}
For every $r$ such that $\frac{r^2}{4}$ is a regular value of f, $D_r$ has smooth boundary. By the Stokes theorem,  we get
\begin{eqnarray*}
I^\prime(t)&\leq & -t^{-\kappa(r)-1}\int_{\partial D_r} \left\langle e^{-\frac{f}{t}}\nabla f, \frac{\nabla f}{|\nabla f|}\right\rangle dv\\
		   &\leq & -t^{-\kappa(r)-1}\int_{\partial D_r} e^{-\frac{f}{t}}|\nabla f|dv \leq 0.
\end{eqnarray*}
Integrating $I^\prime(t)$ over $t$, from $1$ to $e^{\varepsilon}$, where $\varepsilon>0$ is arbitrary, we obtain $I(e^{\varepsilon})\leq I(1)$, that is
\stepcounter{thm}
\begin{eqnarray}
e^{-\varepsilon\kappa(r)}\int_{\bar{D}_r} e^{-\frac{f}{e^{\varepsilon}}}dv\leq \int_{\bar{D}_r} e^{-f}dv.
\end{eqnarray}
By the monotone convergence theorem, the inequality above holds for any $r>0$. Since $2\sqrt{f}\leq r$ on $\bar{D}_r$, we have
\stepcounter{thm}
\begin{eqnarray}\label{4}
e^{-\varepsilon\kappa(r)}e^{-\frac{r^2}{4e^{\varepsilon}}}\int_{\bar{D}_r} dv\leq \int_{\bar{D}_r} e^{-f}dv.
\end{eqnarray}
Moreover,
\stepcounter{thm}
\begin{eqnarray}\label{5}
\int_{\bar{D}_r} e^{-f}dv-\int_{\bar{D}_{r-1}} e^{-f}dv= \int_{\bar{D}_r\setminus \bar{D}_{r-1}} e^{-f}dv\leq e^{-\frac{(r-1)^2}{4}}\int_{\bar{D}_r}dv.
\end{eqnarray}
Combining (\ref{4}) and (\ref{5}), we get
\stepcounter{thm}
\begin{eqnarray}\label{6}
\int_{\bar{D}_r} e^{-f}dv-\int_{\bar{D}_{r-1}} e^{-f}dv\leq e^{\varepsilon\kappa(r)+\frac{r^2}{4e^{\varepsilon}}-\frac{(r-1)^2}{4}}\int_{\bar{D}_r} e^{-f}dv.
\end{eqnarray}
Since $\kappa(r)=a_2r^2+a_1r+a_0$, there exist $r_0\in \mathbb{R}$ such that for $r\geq r_0$ and $\varepsilon_0$ sufficiently small
\begin{eqnarray*}
e^{\varepsilon_0(a_2r^2+a_1r+a_0)+\frac{r^2}{4e^{\varepsilon_0}}-\frac{(r-1)^2}{4}}<e^{-r}.
\end{eqnarray*}
From (\ref{6}), we have
\begin{eqnarray*}
\int_{\bar{D}_r} e^{-f}dv\leq \frac{1}{1-e^{-r}}\int_{\bar{D}_{r-1}} e^{-f}dv.
\end{eqnarray*}
Then for any integer N
\begin{eqnarray*}
\int_{\bar{D}_{r_0+N}} e^{-f}dv\leq \left(\prod^N_{i=0} \frac{1}{1-e^{-r_0-i}}\right)\int_{\bar{D}_{r_0-1}} e^{-f}dv<\infty,
\end{eqnarray*}
which implies that $\int_{M} e^{-f}dv<+\infty$. Moreover, from (\ref{4}) we obtain
\begin{eqnarray*}
e^{-\varepsilon\kappa(r)}e^{-\frac{r^2}{4e^{\varepsilon}}}\int_{\bar{D}_r} dv\leq \
int_{\bar{D}_r} e^{-f}dv\leq \int_{M} e^{-f}dv<\infty.
\end{eqnarray*}
Therefore
\begin{eqnarray*}
V(r)\leq Ce^{\varepsilon\kappa(r)+\frac{r^2}{4e^{\varepsilon}}}.
\end{eqnarray*}
\end{proof}
\rem An immediate consequence is that under the same hypothesis of Theorem \ref{0.1}, for $\varepsilon=0$ we have the following volume estimate
\[
V(r)\leq Ce^{\frac{r^2}{4}}.
\]
In the following result, we obtain a volume estimate for submanifolds with weighted mean curvature of at most linear growth. In particular, we obtain an explicit estimate for the volume of self-expanders and translating solitons.
\begin{cor}\label{0.2} Let $\Sigma^n\subset\mathbb{R}^{n+p}$ be a complete submanifold such that
\[
|\overrightarrow{H}_f|\leq a_1r+a_0\ \text{on}\ \Sigma\cap B_r(0),\ \forall r>0
\]
where $a_0,a_1\geq 0$. If $\Sigma$ is properly immersed on $\mathbb{R}^{n+p}$, then 
\begin{eqnarray*}
V(B_r(0)\cap \Sigma)\leq C e^{\varepsilon\left(\frac{a_1r^2+a_0r+n}{2}\right)+\frac{r^2}{4e^{\varepsilon}}}
\end{eqnarray*}
for $\varepsilon>0$ arbitrary.
\end{cor}
\begin{proof} Let us verify that $f$ satisfies the conditions of Theorem \ref{0.1}. Indeed
\begin{eqnarray*}
f-|\nabla f|^2=\frac{|x^\perp|^2}{4}\geq 0
\end{eqnarray*}
and
\begin{eqnarray*}
\Delta_ff+f &=& \frac{n}{2}+\langle \overrightarrow{H}, \overline{\nabla}f\rangle - |\nabla f|^2+f\\
            &=& \frac{n}{2}+\langle \overrightarrow{H_f}-(\overline{\nabla}f)^\perp, (\overline{\nabla}f)^\perp\rangle - |\nabla f|^2+f\\
            &=& \frac{n}{2}+\langle\overrightarrow{H_f}, (\overline{\nabla}f)^\perp\rangle.
\end{eqnarray*}
Therefore
\[
\Delta_ff+f\leq \frac{a_1r^2+a_0r+n}{2}
\]
on $\Sigma\cap B_r(0)$. Since $\Sigma$ is properly immersed, it follows that f is proper on $\Sigma$. Applying Theorem \ref{0.1}, we obtain $\int_\Sigma e^{-f}<\infty$ and
\begin{eqnarray*}
V(B_r(0)\cap \Sigma)\leq C e^{\varepsilon\left(\frac{a_1r^2+a_0r+n}{2}\right)+\frac{r^2}{4e^{\varepsilon}}}
\end{eqnarray*}
for $\varepsilon>0$ arbitrary.
\end{proof}
Using the corollary above, we will prove Theorem \ref{ti2}.

\noindent\textit{Proof of Theorem \ref{ti2}}.
From Corollary \ref{0.2}, we get that $(i)$ implies $(ii)$. To prove that $(ii)$ implies $(iii)$, we provide the following estimate
\begin{eqnarray*}
\int_\Sigma e^{\frac{-|x|^2}{4}}dv &\leq & \sum_{j=1}^\infty \int_{\Sigma \cap {B_j\setminus B_{j-1}}}e^{\frac{-|x|^2}{4}}dv\\
								   &\leq & \sum_{j=1}^\infty e^{\frac{-(j-1)^2}{4}}V(\Sigma\cap B_j)\\
								   &\leq & C\sum_{j=1}^\infty e^{\frac{-(j-1)^2}{4}}e^{\overline{a}_2r^2+\overline{a}_1r+\overline{a}_0}\\
								   &=& C\sum_{j=1}^\infty e^{\frac{(4\overline{a}_2-1)j^2+(2+4\overline{a}_1)j+(4\overline{a}_0-1)}{4}}.
\end{eqnarray*}
Since $\overline{a}_2<\frac{1}{4}$, the right side of the inequality converges. Therefore
\[
\int_\Sigma e^{\frac{-|x|^2}{4}}dv<\infty.
\]
Finally, we prove that $(iii)$ implies $(i)$. Indeed, suppose $\Sigma$ is not proper. Then there exists $r_0$ such that $\bar{B}_{r_0}\cap \Sigma$ is not compact on $\Sigma$. Thus, for a positive constant $a$, there exists a sequence $\{p_k\}$ on $\bar{B}_{r_0}\cap \Sigma$ such that $d_\Sigma(p_k,p_j)\geq a$. Therefore, $B^\Sigma(p_k,\frac{a}{2})\cap B^\Sigma(p_j,\frac{a}{2})=\emptyset$, where $B^\Sigma$ is the geodesic ball on $\Sigma$. Choosing $0<a<\text{min}\{2r_0,\frac{n}{(2a_1+1)r_0+a_0}\}$, for all $p\in B^\Sigma(p_k,\frac{a}{2})$
\begin{eqnarray*}
|p|\leq |p-p_k|+|p_k|\leq d_\Sigma(p,p_k)+|p_k|\leq 2r_0
\end{eqnarray*}
which implies that $B^\Sigma(p_k,\frac{a}{2})\subset B(0,2r_0)$ for all $k$.\\
Since $\Sigma$ is a CWMC hypersurface, for all $p\in \Sigma\cap B(0,2r_0)$
\begin{eqnarray*}
|\overrightarrow{H}|(p)&\leq & |\overrightarrow{H_f}(p)|+\frac{|p^\perp|}{2}\\
					   &\leq & (2a_1+1)r_0+a_0.
\end{eqnarray*}
Considering $r_k:B^\Sigma(p_k,\frac{a}{2})\rightarrow \mathbb{R}$, where $r_k(x)=|x-p_k|$, we have
\begin{eqnarray*}
\Delta r_k^2&=&2n+\langle \overrightarrow{H},\nabla r_k^2\rangle\\
			&\geq & 2n-2|\overrightarrow{H}|r_k\\
			&\geq & 2n-2(a_0+(2a_1+1)r_0)r_k.
\end{eqnarray*}
By the divergence theorem, since $a\leq \frac{n}{(2a_1+1)r_0+a_0}$, for all $0<r<\frac{a}{2}$
\begin{eqnarray*}
\int_{B^\Sigma (p_k,r)}[2n-2a_0r_k-2(2a_1+1)r_0r_k]dv&\leq & \int_{B^\Sigma (p_k,r)} \Delta r_k^2 dv\\
			   &=& \int_{\partial B^\Sigma (p_k,r)}\langle \nabla r_k^2,\nu\rangle dv\\
			   &\leq & 2r A(r)
\end{eqnarray*}
where $\nu$ is the outward normal vector of $\partial B^\Sigma (p_k,r)$ and $A(r)$ the area of $\partial B^\Sigma (p_k,r)$. Using the co-area formula, we obtain
\begin{eqnarray*}
\int_0^r[n-a_0s-(2a_1+1)r_0s]A(s)ds&=&\int_0^r\int_{d_\Sigma(x,p_k)=s}[n-a_0r_k-(2a_1+1)r_0r_k]dv\\
					 &=&\int_{B^\Sigma (p_k,r)} [n-a_0r_k-(2a_1+1)r_0r_k]dv\\
					 &\leq & rA(r).
\end{eqnarray*}
Therefore,
\begin{eqnarray*}
\left(\frac{n}{r}-a_0-(2a_1+1)r_0\right)\leq \frac{V'(r)}{V(r)}.
\end{eqnarray*}
Integrating from $\varepsilon>0$ to $r$, we obtain
\begin{eqnarray*}
\log\left(\frac{r}{\varepsilon}\right)^n-(a_0+(2a_1+1)r_0)(r-\varepsilon)\leq \log\frac{V(r)}{V(\varepsilon)}.
\end{eqnarray*}
Thus
\begin{eqnarray*}
r^ne^{-(a_0+2(a_1+1)r_0)(r-\varepsilon)}\frac{V(\varepsilon)}{\varepsilon^n}\leq V(r).
\end{eqnarray*}
Since
\begin{eqnarray*}
\lim_{\varepsilon\rightarrow 0} \frac{V(\varepsilon)}{\varepsilon^n}=\omega_n,
\end{eqnarray*}
for any $0<r\leq \frac{a}{2}$, we have
\begin{eqnarray*}
V(r)\geq r^n\omega_ne^{-(a_0+2(a_1+1)r_0)r}.
\end{eqnarray*}
Finally, considering that $B^\Sigma(p_k,\frac{a}{2})\cap B^\Sigma(p_j,\frac{a}{2})=\emptyset$ for $k\neq j$, $B^\Sigma(p_k,\frac{a}{2})\subset B(0,2r_0)$ for all $k$, and the inequality obtained above, we have
\begin{eqnarray*}
\int_\Sigma e^{-\frac{|x|^2}{4}}dv     &\geq& \sum_{k=1}^\infty \int_{B^\Sigma(p_k,\frac{a}{2})}e^{-\frac{|x|^2}{4}}dv\\
&\geq & e^{-r_0^2}\sum_{k=1}^\infty V\left(\frac{a}{2}\right)=+\infty,
\end{eqnarray*}
which is a contradiction.
\begin{flushright}
$\Box$
\end{flushright}
An immediate consequence of the theorem above is as follows.
\begin{cor}\label{c1} For any complete n-dimensional CWMC hypersurface $\Sigma$ in $\mathbb{R}^{n+1}$, the following statements are equivalent:\\
(i) $\Sigma$ properly immersed on $\mathbb{R}^{n+1}$;\\
(ii) There exist constants $a_2,a_1,a_0$, with $a_2<\frac{1}{4}$, and $C>0$ such that
\begin{eqnarray*}
V(B_r(0)\cap \Sigma)\leq Ce^{a_2r^2+a_1r+a_0};
\end{eqnarray*}
(iii)$\int_\Sigma e^{-f}<\infty$.
\end{cor}
\section{Application of the main results}\label{s5}
To prove an application of the main theorems of this paper, we need the following lemma:

\begin{lema}\label{l2}Let $\Sigma\subset \mathbb{R}^{n+1}$ be a complete CWMC hypersurface properly embedded such that \\
$H-\lambda>0$. If
\[
\lambda\left(\text{tr}A^3(H-\lambda)+\frac{|A|^2}{2}\right)\leq 0,
\]
then $\int_\Sigma |A|^2e^{-f}<\infty.$
\end{lema}
\begin{proof}
Let us compute $\mathcal{L}(\log(H-\lambda))$,
\begin{eqnarray*}
\Delta \log(H-\lambda)&=& \text{div}(\nabla\log(H-\lambda))\\
					  &=& \text{div}\left(\frac{1}{H-\lambda}\nabla (H-\lambda)\right)\\
					  &=& \frac{1}{H-\lambda}\Delta (H-\lambda) - \frac{1}{(H-\lambda)^2}|\nabla(H-\lambda|^2.
\end{eqnarray*}
From Lema \ref{l1}, we get
\begin{eqnarray*}
\mathcal{L}(\log(H-\lambda))&=& \frac{1}{H-\lambda}\mathcal{L}(H-\lambda)-|\nabla \log (H-\lambda)|^2\\
							&=& \frac{1}{2}-|A|^2+\frac{\lambda}{2(H-\lambda)}-|\nabla \log (H-\lambda)|^2.
\end{eqnarray*}
Considering $\eta\in C_0^\infty(\Sigma)$ and integrating the equation above we have, 
\begin{eqnarray*}
\int_\Sigma \eta^2\left(|A|^2 -\frac{1}{2}-\frac{\lambda}{2(H-\lambda)}+|\nabla \log (H-\lambda)|^2\right) e^{-f}&=&-\int_\Sigma \eta^2 \mathcal{L}(\log(H-\lambda))e^{-f}\\
      &=&\int_\Sigma \langle \nabla \eta^2, \nabla \log (H-\lambda)\rangle e^{-f}\\
      &\leq& \int_\Sigma \left(|\nabla \eta|^2+\eta^2|\nabla \log(H-\lambda)|^2\right)e^{-f}.
\end{eqnarray*}
Therefore,
\[
\int_\Sigma \eta^2 |A|^2e^{-f}\leq \int_\Sigma \left(|\nabla \eta|^2+\frac{\eta^2}{2}+\frac{\lambda\eta^2}{2(H-\lambda)}\right)e^{-f}.
\]
For $\lambda<0$, we get
\[
\int_\Sigma \eta^2 |A|^2e^{-f}\leq \int_\Sigma \left(|\nabla \eta|^2+\frac{\eta^2}{2}\right)e^{-f}.
\]
Let us consider a sequence $\eta_k\in C_0^\infty(\Sigma)$, such that $\eta_k=1$ in $B_k^\Sigma(p)$, $\eta_k=0$ in $\Sigma\setminus B_{k+1}^\Sigma(p)$ and $|\nabla \eta_k|\leq 1$ for every $k$. By the monotone convergence theorem and the condition (iii) in Corollary \ref{c1}, we can conclude the proof for this case.\\
When $\lambda>0$, by the hypothesis we have
\[
\frac{\lambda}{H-\lambda}\leq -\frac{2\lambda \text{tr}A^3}{|A|^2}\leq 2|\lambda||A|\leq |\lambda|\left(\frac{|A|^2}{\delta}+\delta\right).
\]
Therefore,
\[
\int_{\Sigma}\left(1-\frac{|\lambda|}{2\delta}\right)\eta^2|A|^2 e^{-f}\leq \int_\Sigma\left(|\nabla \eta|^2+\left(1+\frac{|\lambda|\delta}{2}\right)\eta^2\right) e^{-f}.
\]
for any $\delta>0$. For $\delta$ sufficiently large, we get
\[
\int_\Sigma\eta^2|A|^2 e^{-f}\leq \frac{1+\frac{|\lambda|\delta}{2}}{1-\frac{|\lambda|}{2\delta}}\int_\Sigma \left(|\nabla \eta|^2+\frac{\eta^2}{2}\right)e^{-f}.
\]
Using the same argument as before, we conclude the proof of the lemma.
\end{proof}
As a corollary of Theorem \ref{ti1} and the lemma above, we prove Cheng-Wei classification theorem.
\begin{cor} Let $\Sigma\subset \mathbb{R}^{n+1}$ be a complete properly embedded CWMC hypersurface. If $H-\lambda\geq 0$ and $\lambda\left( \text{tr}A^3(H-\lambda)+\frac{|A|^2}{2}\right)\leq 0$, then $\Sigma$ must be either a hyperplane or $S^{k}_r(0)\times \mathbb{R}^{n-k}$, $1\leq k \leq n$.
\end{cor}
\begin{proof}
To prove this corollary, we only need to see that the condition (iii) is satisfied in Theorem \ref{ti1}. Hence
\[
\frac{1}{k^2}\int_{B^\Sigma_{2k}(p)\setminus B^\Sigma_{k}(p)} |A|^2 e^{-f}\leq \frac{1}{k^2}\int_{\Sigma} |A|^2 e^{-f}.
\]
From Lemma \ref{l2}, $\int_{\Sigma} |A|^2 e^{-f}<\infty$. Therefore, when $k\rightarrow 0$, we have
\[
\frac{1}{k^2}\int_{B^\Sigma_{2k}(p)\setminus B^\Sigma_{k}(p)} |A|^2 e^{-f}\rightarrow 0
\]
concluding the proof.
\end{proof}

\bibliographystyle{amsalpha}
\bibliography{Biblio}

\newcommand{\etalchar}[1]{$^{#1}$}
\providecommand{\bysame}{\leavevmode\hbox to3em{\hrulefill}\thinspace}
\providecommand{\MR}{\relax\ifhmode\unskip\space\fi MR }
\providecommand{\MRhref}[2]{%
  \href{http://www.ams.org/mathscinet-getitem?mr=#1}{#2}
}
\providecommand{\href}[2]{#2}
\begin{thebibliography}{DX{\etalchar{+}}13}

\bibitem[AR18]{alencar2018stability}
Hil{\'a}rio Alencar and Adina Rocha, \emph{Stability and geometric properties
  of constant weighted mean curvature hypersurfaces in gradient ricci
  solitons}, Annals of Global Analysis and Geometry (2018), 1--21.

\bibitem[CM12]{colding2012generic}
Tobias~H Colding and William~P Minicozzi, \emph{Generic mean curvature flow i;
  generic singularities}, Annals of Mathematics (2012), 755--833.

\bibitem[CVZ19]{cheng2019volume}
Xu~Cheng, Matheus Vieira, and Detang Zhou, \emph{Volume growth of complete
  submanifolds in gradient ricci solitons with bounded weighted mean
  curvature}, arXiv preprint arXiv:1909.05724 (2019).

\bibitem[CW18]{cheng2018complete}
Qing-Ming Cheng and Guoxin Wei, \emph{Complete $\lambda$-hypersurfaces of
  weighted volume-preserving mean curvature flow}, Calculus of Variations and
  Partial Differential Equations \textbf{57} (2018), no.~2, 32.

\bibitem[CZ13]{cheng2013volume}
Xu~Cheng and Detang Zhou, \emph{Volume estimate about shrinkers}, Proceedings
  of the American Mathematical Society \textbf{141} (2013), no.~2, 687--696.

\bibitem[DX{\etalchar{+}}13]{ding2013volume}
Qi~Ding, YL~Xin, et~al., \emph{Volume growth eigenvalue and compactness for
  self-shrinkers}, Asian Journal of Mathematics \textbf{17} (2013), no.~3,
  443--456.

\bibitem[Gua18]{guang2018gap}
Qiang Guang, \emph{Gap and rigidity theorems of $\lambda$-hypersurfaces},
  Proceedings of the American Mathematical Society \textbf{146} (2018), no.~10,
  4459--4471.

\bibitem[Hei17]{heilman2017symmetric}
Steven Heilman, \emph{Symmetric convex sets with minimal gaussian surface
  area}, arXiv preprint arXiv:1705.06643 (2017).

\bibitem[Hui]{huisken54local}
Gerhard Huisken, \emph{Local and global behaviour of hypersurfaces moving by
  mean curvature. differential geometry: partial differential equations on
  manifolds (los angeles, ca, 1990), 175--191}, Proc. Sympos. Pure Math,
  vol.~54.

\bibitem[Hui90]{huisken1990asymptotic}
\bysame, \emph{Asymptotic-behavior for singularities of the mean-curvature
  flow}, Journal of Differential Geometry \textbf{31} (1990), no.~1, 285--299.

\bibitem[Law69]{lawson1969local}
H~Blaine Lawson, \emph{Local rigidity theorems for minimal hypersurfaces},
  Annals of Mathematics (1969), 187--197.

\bibitem[Rim14]{rimoldi2014classification}
Michele Rimoldi, \emph{On a classification theorem for self--shrinkers},
  Proceedings of the American Mathematical Society \textbf{142} (2014), no.~10,
  3605--3613.

\bibitem[TZ17]{tasayco2017uniqueness}
Ditter Tasayco and Detang Zhou, \emph{Uniqueness of grim hyperplanes for mean
  curvature flows}, Archiv der Mathematik \textbf{109} (2017), no.~2, 191--200.

\end{thebibliography}
\end{document}